\documentclass[12pt]{amsart}

\usepackage{color}
\usepackage[centertags]{amsmath}
\usepackage[margin=1in]{geometry}
\usepackage{amsfonts}
\usepackage{amsthm}
\usepackage{newlfont}
\usepackage{verbatim}
\usepackage{amscd}
\usepackage{amsgen}
\usepackage{amssymb}
\usepackage{stmaryrd}
\usepackage{amssymb,amsmath}
\usepackage{amscd}
\usepackage{enumerate}
\usepackage{color}
\usepackage{xypic}
\usepackage{graphicx}

\newlength{\defbaselineskip} 
\theoremstyle{plain}
\newtheorem{thm}{Theorem}[section]
\newtheorem{cor}[thm]{Corollary}
\newtheorem{con}[thm]{Conjecture}
\newtheorem{df}[thm]{Definition}
\newtheorem{lema}[thm]{Lemma}
\newtheorem{obs}[thm]{Proposition}
\newtheorem{exm}[thm]{Example}
\newtheorem{comp}[thm]{Computation}

\newtheorem{rem}[thm]{Remark}
\newtheorem{pr}{Algorithm}
\theoremstyle{definition} 
\theoremstyle{definition} \newtheorem{ex}{Example}[section] %

 \numberwithin{equation}{section}
\def\p{\mathbb{P}}
\def\r{\mathbb{R}}
\def\z{\mathbb{Z}}

\def\Z{\mathbb{Z}}

\def\c{\mathbb{C}}

\DeclareMathOperator{\End}{End}

\def\p{\mathbb{P}}

\def\ob{\begin{obs}}
\def\kob{\end{obs}}
\def\dow{\begin{proof}}
\def\kdow{\end{proof}}
\def\kwadrat{\hfill$\square$}
\def\tw{\begin{thm}}
\def\ktw{\end{thm}}
\def\hip{\begin{con}}
\def\khip{\end{con}}
\def\lem{\begin{lema}}
\def\klem{\end{lema}}
\def\ex{\begin{exm}}
\def\prog{\begin{pr}}
\def\kprog{\end{pr}}
\def\wn{\begin{cor}}
\def\kwn{\end{cor}}
\def\uwa{\begin{rem}}
\def\kuwa{\end{rem}}
\def\kex{\end{exm}}
\def\dfi{\begin{df}}
\def\kdfi{\end{df}}
\definecolor{zielony}{rgb}{0.5, 0.9, 0.1}
\definecolor{czerwony}{rgb}{0.9, 0.2, 0.1}
\definecolor{niebieski}{rgb}{0.3, 0.1, 0.9}

\begin{document}
\title{{Phylogenetic invariants for group-based models}}
\author{Maria Donten-Bury}
\author{Mateusz Micha\l ek}
\thanks{The first author is supported by a grant of Polish MNiSzW (N N201 611 240). The second author is supported by a grant of Polish MNiSzW (N N201 413 539).}
\maketitle

\begin{abstract}
In this paper we investigate properties of algebraic varieties representing group-based phylogenetic models. We propose a method of generating many phylogenetic invariants. We prove that we obtain all invariants for \emph{any} tree for the two-state Jukes-Cantor model. We conjecture that for a large class of models our method can give all phylogenetic invariants for any tree. We show that for 3-Kimura our conjecture is equivalent to the conjecture of Sturmfels and Sullivant \cite[Conjecture 2]{SS}. This, combined with the results in \cite{SS}, would make it possible to determine all phylogenetic invariants for \emph{any} tree for 3-Kimura model, and also other phylogenetic models.
Next we give the (first) examples of non-normal varieties associated to general group-based model for an abelian group. Following Kubjas \cite{kaie} we prove that for many group-based models varieties associated to trees with the same number of leaves do not have to be deformation equivalent.
\end{abstract}

\section{Introduction}

Phylogenetics is a science that tries to reconstruct the history of evolution. It is strongly connected with many branches of mathematics including algebraic geometry. To each possible history of evolution, represented by a tree, one can associate an algebraic variety, whose special points correspond to possible probability distributions on the DNA states of the living species. For a detailed introduction the reader is advised to look in \cite{PS} and for an algebraic point of view in \cite{4aut}.

From the point of view of applications, one is interested in computing phylogenetic invariants that are polynomials defining the variety. It is very hard to find them in general, however for some special models of evolution much progress has been made. In this paper we are mainly dealing with a large class of equivariant models \cite{DK} -- so called general group-based models. The varieties associated to these models have a natural torus action with a dense orbit (see \cite{Hendy1}, \cite{ES}, \cite{mateusz}). The most influential paper in this area is \cite{SS}, where the authors gave the description of the generators of the ideal, assuming that the ideal of the claw tree is known. Unfortunately not much is known\footnote{After the submission of this paper a result concerning set-theoretic generation in some bounded degree was presented in \cite{Draisma2012}.} on the ideals of the claw trees, apart from the case of the two-state Jukes-Cantor model \cite{Sonja}. In particular we do not even know if the degree in which they are generated is bounded while the number of leaves grows to infinity (Conjecture 1 and 2 in \cite{SS}). The description of
 the ideal of the claw tree is also the main missing ingredient in the description of the ideal of any tree for equivariant models \cite[p. 17]{DK}.

In this paper we propose a method of finding the ideals of the claw trees using a geometric approach -- cf. Section \ref{methodgeneration}. We conjecture that the varieties associated to large claw trees are scheme-theoretic intersections of varieties associated to trees of smaller valency.
This would enable generating the ideals recursively. We prove the conjecture for the Jukes-Cantor model \ref{conj_JC}. An interesting fact is that we can show that our conjecture is equivalent to the one made by Sturmfels and Sullivant for the 3-Kimura model, Proposition \ref{eq}.

\emph{For any general group--based model the phylogenetic invariants of degree $d$ for a claw tree with inner vertex of degree $n$ can be explicitly derived from the phylogenetic invariants for trees with vertices of degree at most $n-1$ for $d<<n$. For the details see Section \ref{phyloinv}.
}

We also investigate geometric properties of algebraic varieties representing phylogenetic models. In particular we give an example of a model associated
to an abelian group that gives a non-normal variety.

\emph{The variety associated to any group containing $\z_6$ or $\z_2\times\z_2\times\z_2$ or $\z_4\times\z_2$ or $\z_8$ and the claw tree with three leaves is not normal (Computation \ref{nienormalneobliczenia} and Proposition \ref{nienormalneind}). The variety associated to any abelian group of cardinality at most $5$ or to $\z_7$ and any trivalent tree is normal. The variety associated to the group $\z_2$ and any tree is normal (Proposition \ref{JCnormal}).}

The results on normality are strongly connected to deformation problems. It is well-known that
algebraic varieties representing trivalent trees with the same number of leaves are deformation equivalent for the binary Jukes-Cantor model. The original
geometric proof can be found in \cite{BW} and a new, more combinatorial one, in \cite{Il}. A new result of Kaie Kubjas shows that this is not true
for the 3-Kimura model \cite{kaie}. Our results are as follows.

\emph{The varieties associated to two different trivalent trees with the same number of leaves -- the caterpillar and snowflake -- and one of the groups $\z_3$, $\z_4$, $\z_5$, $\z_7$ or $\z_2\times \z_2$ have different Hilbert polynomials (Computation \ref{2kimura_poly}).}

The idea of the proof is to calculate the number of integer points in $nP$, where $P$ is the polytope associated
to the algebraic model of a phylogenetic tree. This task is much easier for normal varieties (in this case we obtain Hilbert-Ehrhart polynomial of the algebraic model).

One of the tools that we use is a program that computes the polytope defining a toric variety for a given tree and a group (see section \ref{program}). Our program, implementing the algorithm described in \cite{mateusz}, can be found at \texttt{http://www.mimuw.edu.pl/\~{}marysia/polytopes} (with a detailed instruction and specification of the input and output data format). It can be used to print the vertices of the polytope associated to a tree given by the user in an input file and one of the groups (with small numbers of elements) defined in the source code, but also by a slight modification of a source code one can obtain polytopes associated to models with other groups action.

\section*{Acknowledgements}
The authors would like to thank very much Jaros\l aw Wi\'sniewski for introducing them to the topic and encouraging them to study it further. We thank Winfried Bruns for the help with computations.
We would also like to thank Elizabeth Allman and Laurent Manivel for useful remarks.\\
The first author is supported by a grant of Polish MNiSzW (N N201 611 240).\\
The second author is supported by a grant of Polish MNiSzW (N N201 413 539).

\section{General group-based models}
To a tree $T$, a vector space $W$ with a distinguished basis and a subspace $\widetilde W\subset \End(W)$ one can associate an algebraic variety $X(T,\widetilde W)$. People familiar with phylogenetics should identify basis elements of $W$ and states of investigated random variables.
By \emph{a model} we mean a choice of $\widetilde W$. The construction of the variety $X(T,\widetilde W)$ is well-known for specialist and not necessary in general to understand the main ideas of the paper. It can be found for example in \cite{4aut}. Thus we focus on a class of general group-based models and describe the construction in this setting.
All the trees that we consider are rooted and we assume that the edges are directed away from the root.
\begin{df}[General group-based model]
Let $W$ be the regular representation of a finite, abelian group $G$. For a general group-based model we define the subspace $\widetilde W:=\End(W)^G$ as $G$ invariant endomorphisms of $W$.
\end{df}
\begin{exm}\emph{
For the group $\z_2$ we obtain the two-state Jukes-Cantor model also known as the Cavender-Farris-Neyman model. The elements of $\widetilde W$ represented as matrices are of the form:}
    \[
\left[
\begin{array}{cccccccc}
a&b\\
b&a\\
\end{array}
\right].
\]

\emph{
For the group $\z_2\times\z_2$ we obtain the 3-Kimura model. The elements of $\widetilde W$ represented as matrices are of the form:}
\[
\left[
\begin{array}{cccccccc}
a&b&c&d\\
b&a&d&c\\
c&d&a&b\\
d&c&b&a\\
\end{array}
\right].
\]
\end{exm}
As $\widetilde W$ is defined by the group $G$ we will write $X(T,G)$ instead of $X(T,\widetilde W)$.
In this case we can easily describe the affine cone over the variety $X(T,G)$. It is a spectrum of the semigroup algebra, with a semigroup generated by the lattice points of a polytope. The construction of the polytope depends on the tree and the group. It is purely combinatorial and very easy -- see Algorithm \ref{algorytm}. The description of the defining polytope is also given in Theorem \ref{mateusz}. Sometimes we do not specify the model and denote the associated variety $X(T)$.

The fact that to a general group-based models one can associate a (not necessary normal) toric variety was observed by many authors \cite{ES}, \cite{ESS}, \cite{SS}. To describe the defining polytope we need the following definitions.
\begin{df}[Group based flows, sockets]\label{sockets}
A group based flow is a function $n: E\rightarrow G$, where $E$ is the set of edges of the tree. Moreover we require that for any inner vertex $v$ of the tree, $e_0$ an incoming edge and $e_1,\dots, e_k$ outgoing edges we have
$$n(e_0)=n(e_1)+\dots+n(e_k).$$
A socket is a function $s: L\rightarrow G$, where $L$ is the set of leaves of the tree, that is edges adjacent to vertices of degree one. Moreover we require that the sum $\sum_{l\in L} s(l)$ is equal to the neutral element of the group.
\end{df}
Let us explain the terminology. A group based flow assigns elements of a group to edges, with a condition that for each inner vertex the sum of elements associated to incoming edges equals the sum of those associated to outgoing edges. If the group was equal to $\r$ this would be the condition for a flow, where leaves are sources and sinks. We use the terminology of group based flows, as this is a direct generalization to any abelian group of the notion introduced in \cite{BW} for $\z_2$. As all the flows that we use are group based, we write flow, meaning a group based flow.

\begin{exm}\label{socket}\emph{
Consider the group $G=\z_3$ and the following tree:
$$
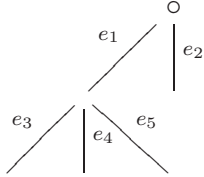
\xymatrix{
&&\circ\ar@{-}[dl]_{e_1}\ar@{-}[d]^{e_2}\\
&\ar@{-}[dl]_{e_3}\ar@{-}[d]^{e_4}\ar@{-}[dr]^{e_5}&\\
&&\\
}$$
Here $e_2$, $e_3$, $e_4$ and $e_5$ are leaves. An example of a socket is an association $e_2\rightarrow 1$, $e_3\rightarrow 1$, $e_4\rightarrow 2$, $e_5\rightarrow 2$.}

\emph{We can make a flow using the same association and extending it by $e_1\rightarrow 2$.}
\end{exm}

By restricting a flow to leaves we get a socket and this map is in fact a bijection.

Let us present a combinatorial description of the polytope that represents this toric variety.
\begin{thm}[\cite{mateusz}]\label{mateusz}
Let $P$ be a polytope representing a general group-based model. All integral points of $P$ are vertices. There is a natural bijection between vertices of $P$ and flows.
The polytope $P$ is a subpolytope of a lattice with basis elements indexed by pairs $(e,g)$ where $e$ is an edge of a tree and $g$ a group element.
The vertex of $P$ associated to a flow $n$ is a sum of all basis elements indexed by such pairs $(e,g)$ that satisfy $n(e)=g$.
\kwadrat
\end{thm}

The generating binomials of a toric ideal associated to a polytope $P$ correspond to integral relations between integer points of this polytope -- see for example \cite{Sturmfelsksiazka}, \cite{Cox}, \cite{fulton}. Hence in our situation phylogenetic invariants correspond to relations between flows. In the abelian case each such relation can be described in the following way. We number all edges of a tree from $1$ to $e$. The flows are specific $e$-tuples of group elements. For example for the claw tree these are $e$-tuples of group elements summing up to the neutral element. Each relation of degree $d$ between the flows is encoded as a pair of matrices with $d$ columns and $e$ rows with entries that are group elements. We require that each column represents a flow. Moreover the rows of both matrices are the same up to permutation. For a general group-based model this is a purely combinatorial description of all phylogenetic invariants for any tree \cite{SS}. This
  construction can be generalized to other, so called $G$-models \cite{mateusz}, \cite{WDB}.
\begin{exm}\emph{
Consider the binary Jukes-Cantor model, that is the model corresponding to the group $\z_2$ and the following tree.}
 \begin{equation}\label{drzewo}\end{equation}
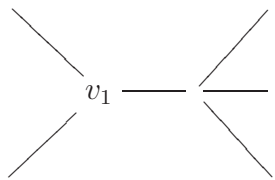
$$\xymatrix{
\ar@{-}[dr]&&&\ar@{-}[dl]\\
&v_1\ar@{-}[r]&\ar@{-}[dr]\ar@{-}[r]&\\
\ar@{-}[ur]&&&\\
}$$
\emph{
The leaves adjacent to $v_1$ will be represented by first two rows. The third row corresponds to the inner edge.  An example of a relation is given by a pair of matrices:}
\[
\left[
\begin{array}{cccccccc}
1&0\\
0&1\\
1&1\\
1&0\\
0&1\\
0&0\\
\end{array}
\right],
\left[
\begin{array}{cccccccc}
0&1\\
1&0\\
1&1\\
1&0\\
0&1\\
0&0\\
\end{array}
\right].
\]\emph{
The numbers $0$ and $1$ are treated as elements of $\z_2$. Due to the definition of the socket the third row has to be the sum of both the first two and last three rows.}
\end{exm}


\subsection{Algorithm of finding the polytope}\label{program}

The first step of all our computations is passing from an abstract model to the lattice polytope associated to it. The following algorithm for computing the associated polytope was proposed in \cite[Sect. 4]{mateusz}. Let $E$ be the set of edges of the tree and $N$ be the set of its inner vertices.
\begin{pr}\label{algorytm}

\begin{enumerate}
\item Orient the edges of the tree from the root.
\item For each inner vertex choose one outgoing edge.
\item Make a bijection $b:G\rightarrow B\subset \z^{|G|}$, where $B$ is the standard basis of $\z^{|G|}$.
\item Consider all possible associations of elements of $G$ with not-chosen edges (there are $|G|^{|E|-|N|}$ such associations).
\item For each such associations, make a full association by assigning an element of $G$ to each chosen edge in such a way that the (signed) sum of elements around each inner vertex gives a neutral element in $G$.
\item For each full association output the vertex of the polytope: $(b(g_e)_{e\in E})$, where $g_e$ is the element of the group associated to edge $e$.
\end{enumerate}
\end{pr}

There are two non obvious points in the implementation. One is step~2 of the algorithm: making a choice of an outgoing edge from each vertex (the tree is rooted and the edges are directed from the root). It is much easier to choose incoming edge for each vertex except the root, as this choice is almost canonical (depends only on the rooting), so does not have to be stored in the memory. By precomputing the group operations and storing the result we obtain the complexity $O(|N||G|^{|E|-|N|})$, as predicted in \cite{mateusz}. The current version of the program operates on the abelian groups defined in the source.

As a result we have a fast program which takes a tree in a simple text format as an input and allows to choose one of the groups from the library. It computes the list of vertices of a polytope associated to the input model and outputs it to a file. It also enables the user to work with this polytope, given as an object of an inner class of the program, in the further computations. For example, it simplified significantly the programming necessary to perform the computations of Hilbert-Ehrhart polynomials, described in \ref{HEpoly}.

\section{Phylogenetic invariants}\label{phyloinv}
In this section we investigate the most important objects of phylogenetic algebraic geometry -- ideals of phylogenetic invariants. The main problem in this area is to give an effective description of the whole ideal of the variety associated to a given model on a tree.
Our task is to find an efficient way to compute generators of these ideals.

We suggest a way of obtaining all phylogenetic invariants of a claw tree of a general group-based model -- more precisely we conjecture that our invariants generate the
whole ideal of the variety. These, together with the results of \cite{SS} could provide an algorithm listing all generators of the ideal of
phylogenetic invariants for any tree and for any general group-based model.

\subsection{Inspirations}
The inspirations for our method were the conjectures made by Sturmfels and Sullivant in \cite{SS}.
They are still open but, as we will see, they strongly support our ideas. In particular, we will prove in Proposition \ref{eq} that our algorithm works for trees with more than $8$ leaves for the 3-Kimura model if we assume that the weaker conjecture made in \cite{SS} holds.

First we introduce some notation. Let $K_{n,1}$ be a claw tree with $n$ leaves. Let $G$ be a finite abelian group. Let $\phi(G,n)=d$
be the least natural number such that the ideal associated to $K_{n,1}$ for the general group-based model $G$ is generated in degree $d$. The phylogenetic complexity of the group $G$ is defined as $\phi(G)=sup_n\phi(G,n)$. Based on numerical results Sturmfels and Sullivant suggested the following conjecture:
\begin{con}\label{stop}

For any abelian group $G$ we have $\phi(G)\leq |G|$.
\end{con}
This conjecture was separately stated for the 3-Kimura model, that is for $G=\z_2\times\z_2$.

Still very little is known about the function $\phi$ apart from the case of the binary Jukes-Cantor model (see also \cite{Sonja}):
 \begin{obs}[Sturmfels, Sullivant]\label{dwa}
 In case of the binary Jukes-Cantor model $\phi(\z_2)=2$.\kwadrat
 \end{obs}

There are also some computational results -- to the table in \cite{SS} presenting the computations made by Sturmfels and Sullivant a few cases can be added.

\begin{comp}
Using \texttt{4ti2} software \cite{4ti2} we obtained the following:
\begin{itemize}
\item $\phi(6,\z_3) = 3$,
\item $\phi(4,\z_5) = 4$,
\item $\phi(3,\z_8) = 8$,
\item $\phi(3,\z_2\times \z_2 \times \z_2) = 8$,
\item $\phi(3,\z_4 \times \z_2) = 8$.
\end{itemize}
\end{comp}

 For the 3-Kimura model we do not even know whether the function $\phi$ is bounded. As we will see later, this conjecture is strongly related to the one stated in the next section.

\subsection{A method for obtaining phylogenetic invariants}\label{methodgeneration}

We propose a method that is inspired by the geometry of the varieties we consider. First we have to introduce some notation.
\begin{df}[contraction, prolongation]
We say that a tree $T_1$ is obtained by a contraction of an edge $e$ of the tree $T_2$ if the vertices of $T_1$ correspond to vertices of $T_2$ with the two vertices adjacent to $e$ identified. Notice that there is a bijection between edges of $T_2$ different from $e$ and edges of $T_1$.

In such a situation we say that $T_2$ is a prolongation of $T_1$.
\end{df}
$$
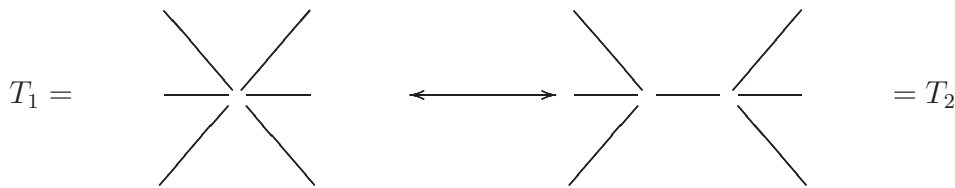
{\xymatrix{
&\textbf{\ar@{-}[dr]}&&\textbf{\ar@{-}[dl]}&&&\textbf{\ar@{-}[dr]}&&&&\\
T_1=&&\textbf{\ar@{-}[dr]}\textbf{\ar@{-}[r]}\textbf{\ar@{-}[l]}\textbf{\ar@{-}[dl]}&&\textbf{\ar@{<->}[rr]}&&&\textbf{\ar@{-}[r]}\textbf{\ar@{-}[l]}\textbf{\ar@{-}[dl]}&\textbf{\ar@{-}[r]}\textbf{\ar@{-}[dr]}\textbf{\ar@{-}[ur]} &&=T_2\\
&&&&&&&&&&\\}}$$
\begin{rem}
Note that these definitions are not the same as the definitions of flattenings introduced in \cite{AllRhMarkov} and further studied in \cite{DK}.
\end{rem}

Assume that we are in an abelian case.
Using Algorithm \ref{algorytm} and Definition \ref{sockets} one can see that vertices of the polytope correspond to sockets. On the other had it is well known that vertices of the polytope correspond to coordinates of the ambient space of the variety associated to the polytope.
In this setting the variety $X(T_1)$ associated to the tree $T_1$ is in a natural way a subvariety of $X(T_2)$, for any model. Notice that we can identify sockets of both varieties, as we may identify their leaves, so both varieties are contained in $\p^s$, where $s$ is the number of sockets. The natural inclusion corresponds to the projection of character lattices: we forget all the coordinates corresponding to the edge joining the vertices $v_1$ and $v_2$. Now the following conjecture is natural:
\begin{con}\label{glhip}
 The variety $X(K_{n,1})$ is equal to the (scheme theoretic) intersection of all the varieties $X(T_i)$, where $T_i$ is a prolongation of $K_{n,1}$ that has only two inner vertices, both of them of valency at least three.
\end{con}
As $X(K_{n,1})$ is a subvariety of $X(T_i)$ for any prolongation $T_i$ one inclusion is obvious.
 Note also that the valency condition is made, because otherwise the conjecture would be obvious -- one of the varieties that we intersect would be
 equal to $X(K_{n,1})$ (contraction of a vertex of degree 2 does not change the corresponding variety). All $T_i$ have strictly
 smaller maximal valency than $K_{n,1}$, so if the conjecture holds then we can inductively use Theorem 23 of Sturmfels and Sullivant \cite{SS} (see
 also Theorem 12 \cite{sull}) to obtain all phylogenetic invariants for a given model for any tree of any valency, knowing just the ideal of the $K_{3,1}$.
 In such a case the ideal of $X(K_{n,1})$ is just the sum of ideals of trees with smaller valency. More precisely, if \ref{glhip} holds then the degree in which the ideals of claw trees are generated cannot grow when the number of leaves gets bigger. This means that $\phi(G)=\phi(G,3)$ which can be computed in many cases. In particular, Conjecture \ref{glhip} implies all cases of Conjecture \ref{stop} in which we can compute $\phi(G,3)$ -- this includes the most interesting 3-Kimura model.
 \begin{df}[tripod]
 The tree $K_{3,1}$ will be called a tripod.
 \end{df}
\begin{rem}
A reader familiar with phylogenetics may observe that varieties $X(T_1)$ and $X(T_2)$ are naturally contained in the same ambient space for any model, even if it does not give rise to toric varieties. Thus Conjecture \ref{glhip} can help to compute the ideals of claw trees for a large class of phylogenetic models.
\end{rem}

Of course one may argue that Conjecture \ref{glhip} above is too strong to be true. We prove it for the binary Jukes-Cantor model in Proposition \ref{conj_JC}. We will also consider two modifications of this conjecture to weaker conjectures that can still have a lot of applications. The first modification just states that Conjecture \ref{glhip} holds for $n$ large enough.

\begin{obs}\label{rownowaznosc}
Conjecture \ref{glhip} holds for $n$ large enough if and only if the function $\phi$ is bounded.
\end{obs}
\begin{proof}
One implication is obvious. Suppose that \ref{glhip} holds for $n>n_0$. We choose such $d$ that the ideals associated to $K_{l,1}$ are generated in degree $m$ for $l\leq n_0$. Using \ref{glhip} and the results of \cite{SS} we can describe the ideal associated to $K_{n,1}$ as the sum of ideals generated in degree $m$. It follows that this ideal is also generated in degree $m$, so the function $\phi$ is bounded by $m$.

For the other implication let us assume that $\phi(n)\leq m$. Let us consider any binomial $B$ that is in the ideal of the claw tree and is of degree less or equal to $m$. We prove that $B$ belongs to the ideal of some prolongation of a tree $T$, which is in fact more than the statement of Conjecture \ref{glhip}.

Such a binomial can be described as a linear relation between (at most $m$) vertices of the polytope of this variety. Each vertex is given by an association of orbits of characters to edges such that there exist representatives of orbits that sum up to a trivial character. Let us fix such representatives, so that each vertex is given by $n$ characters summing up to a trivial character.

Now the binomial $B$ can be presented as a pair of matrices $A_1$ and $A_2$ with characters as entries. Each column of the matrices is a vertex of the polytope. The matrices have at most $m$ columns and exactly $n$ rows. Let us consider the matrix $A=A_1-A_2$, that is entries of the matrix $A$ are characters that are differences of entries of $A_1$ and $A_2$. We can subdivide the first column of $A$ into groups of at most $|H|$ elements summing up to a trivial character. Then inductively we can subdivide the rows into groups of at most $|H|^i$ elements summing up to a trivial character in each column up to the $i$-th one.

For $n>|H|^m+1$ we can find a set $S$ of rows of $A$ such that the characters sum up to a trivial character in each column restricted to $S$, such that both the cardinality of $S$ and of its complement are greater then 1.
Note that the sums of the entries lying in a chosen column and in the rows in $S$ are the same in $A_1$ and $A_2$. Therefore, adding to both matrices an extra row whose entries are equal to the sum of the entries in the subset $S$ gives a representation of a binomial $B$ on a prolongation of $T$.
\end{proof}

In particular, this means that if Conjecture \ref{stop} of Sturmfels and Sullivant holds for the 3-Kimura model, then Conjecture \ref{glhip} also holds for this model for $n>257$. Later we will significantly improve this estimation.

For the second modification of Conjecture \ref{glhip} let us recall a few facts on toric varieties. Let $T_1$ and $T_2$ be two tori with lattices of characters given respectively by $M_1$ and $M_2$. Assume that both of them are contained in a third torus $T$ with the character lattice $M$. The inclusions give natural isomorphisms $M_1\simeq M/ K_1$ and $M_2\simeq M/K_2$, where $K_1$ and $K_2$ are torsion free lattices corresponding to characters that are trivial when restricted respectively to $T_1$ and $T_2$. The ideal of each torus (inside the big torus) is generated by binomials corresponding to such trivial characters. The points of $T$ are given by semigroup morphisms $M\rightarrow \c^*$. The points of $T_i$ are those morphisms that associate $1$ to each character from $K_i$. We see that the points of the intersection $T_1\cap T_2$ are those morphisms $M\rightarrow \c^*$ that associate 1 to each character from the lattice $K_1+K_2$. Of course the (possibly reducible) intersection $Y$ is generated by the ideal corresponding to $K_1+K_2$. This lattice may be not saturated, but $Y$ contains a distinguished torus $T'$, that is one of its connected components. If $K'$ is the saturation of the lattice $K_1+K_2$ then the characters of $T'$ are given by the lattice $M/ K'$. Let $X_i$ be the toric variety that is the closure of $T_i$, and $X'$ be the closure of $T'$. We call the toric variety $X'$ the \emph{toric intersection} of $X_1$ and $X_2$.

 In the setting of \ref{glhip} we conjecture the following:
\begin{con}\label{hip2}
The toric variety $X(T)$ is the toric intersection of all the toric varieties $X(T_i)$.
\end{con}

This conjecture differs from the previous one by the fact that we allow the intersection to be reducible, with one distinguished irreducible component equal to $X(T)$.
We state this conjecture, because it can be checked using only the tori. As the biologically meaningful points are contained in the torus (see \cite{CFS}), this conjecture is of much importance for applications. Moreover, it is quite easy to check it for trees with small enough number of leaves using computer programs. To explain it properly, let us consider the following general setting.

Assume that the tori $T_i$ are associated to polytopes $P_i$ and that $T$ is just the torus of the projective space $\p^n\supseteq T_i$. Let $A_i$ be a matrix whose columns represent vertices of the polytope $P_i$. The characters trivial on $T_i$ or respectively binomials generating the ideal of $T_i$ are exactly represented by integer vectors in the kernel of $A_i$. The characters trivial on the intersection are given by integer vectors in $\ker A_1+\ker A_2$.

Note that the ideal of the toric intersection $T'$ of the tori $T_i$ in $T$ is generated by binomials corresponding to characters trivial on $T'$, that is by the saturation of $\ker A_1+\ker A_2$. These binomials define a toric variety in $\p^n$. This variety is contained in the intersection (in fact it is a toric component) of the toric varieties that are the closures of $T_i$. The equality may not hold however, as the intersection might be reducible.

In Conjecture \ref{hip2} we have to compare two tori, one contained in the other. To do this, it is enough to compare their dimension, that is the rank of the character lattice. Let us note that the dimension of the intersection $T_1\cap T_2$  is given by $n$ minus the dimension (as a vector space) of $\ker A_1+\ker A_2$, as it is equal to the rank of the lattice $\z^n\cap (\ker A_1+\ker A_2)$. To compute this dimension it is enough to compute the ranks of matrices $A_1$, $A_2$ and $B$, where $B$ is a matrix obtained by putting $A_1$ under $A_2$ (that is, $\ker B=\ker A_1\cap \ker A_2$). This can be done very easily using GAP (\cite{GAP}).

The results obtained for small trees will be used in the following section.


\subsection{Main Results}
To support Conjecture \ref{glhip} let us consider the case of binary Jukes-Cantor model. This model is well understood -- see for example \cite{BW}, \cite{Sonja}, \cite{SS}. In particular the quadratic Gr\"obner basis was explicitly constructed for any tree in \cite[Proposition 3]{Sonja}.
Now we can prove the following:
\begin{obs}\label{conj_JC}
Conjecture \ref{glhip} holds for the binary Jukes-Cantor model.
\end{obs}
\begin{proof}
We use the same notation as in the proof of Proposition \ref{rownowaznosc}.
From \ref{dwa} we know that $\phi(\z_2)=2$. Let us consider any binomial of degree 2 for a claw tree with $n$ leaves. This is given by a pair of matrices $A_1$, $A_2$ with 2 columns each. Let $A=A_1-A_2$, where the difference uses the group law. We construct a subset $S$ of the set of rows which gives a prolongation of the tree.

By permuting columns of $A_2$ we may assume that the entries in the first row of the matrix $A$ are trivial. Let $A'$ be the matrix obtained by deleting the first row of $A$.
If we have a row $00$ in $A'$ then we are done, so assume there are only  $01$, $10$ and $11$. Notice that $01$ and $10$ cannot occur at all, as $A_1$ and $A_2$ would not have the same rows up to permutation.
For $n>3$ we can take twice $11$ as a strict subset of the set of rows, summing up to zero in each column.
\end{proof}

From the proof above it follows that in fact to obtain the variety of the claw tree for the binary Jukes-Cantor model it is enough to intersect three varieties corresponding just to three subdivisions. This subdivisions correspond to $S$ containing exactly first and second row or first and third, or second and third row.

The following fact is an easy consequence of the results of \cite{Sonja} and \cite{SS}. However, due to lack of references we include it. For the trivalent trees in was proved in \cite[Appendix]{BW}.
\begin{obs}\label{JCnormal}
The toric variety associated to any tree and the binary Jukes-Cantor model is projectively normal for any tree.
\end{obs}
\begin{proof}
By \cite[Lemma 5.1]{mateusz} it is enough to consider claw trees. Now one can apply \cite[Proposition 1.6]{ToricIdealsGeneratedbyQuadraticBinomials}, where the assumptions are satisfied due to \cite[Proposition 3]{Sonja}.
\end{proof}
In general we conjecture the following.
\begin{con}
Consider a finite abelian group $G$. Suppose that the variety associated to the tripod for the general group based model is projectively normal. Than the variety associated to any tree is projectively normal.
\end{con}

Now we prove the following conditional result for the 3-Kimura model:
\begin{obs}\label{eq}
If Conjecture \ref{stop} of Sturmfels and Sullivant holds then Conjecture \ref{glhip} holds for $n>8$.
\end{obs}
\begin{proof}
We use the same notation as in the previous proof, but instead of considering the matrix $A$ (corresponding to a chosen binomial) with $k$ columns and entries from $\z_2\times \z_2$ we assume that it has $2k$ columns and entries from $\z_2$. Let us note that the number of $1$ in each row both in even and odd columns has to be even. This follows from the fact that rows of $A$ are differences of rows that were equal up to permutation. This means that both projections from $\z_2\times \z_2\rightarrow\z_2$ gave rows that were equal up to permutation. The difference of such vectors always has an even number of $1$.

Once again we may assume that the entries in the first row of $A$ are trivial characters, that is they are equal to zero. Let $A'$ be the matrix obtained by deleting the first row of $A$. For each subset of rows of $A'$ we may consider a vector of length equal to the number of columns of $A'$, whose entries are given by sums of characters from the subset. Note that this vector always has an even number of $1$ both in even and odd columns. Because we assume Conjecture \ref{stop}, the matrix $A'$ has at most 8 columns. By Dirichlet's principle, if $n>8$ then we can find two subsets of rows of $A'$ that are not complements of each other, such that their sum vector is the same. If we take a symmetric difference of these subsets, we obtain a strict, nonempty set $S$ of rows of $A'$, summing up in each column to a trivial character. We add the first row of $A$ to $S$ or its complement, so that both sets have more than one element. Thus we obtain a subdivision of the set of rows of
 $A$ such
  that the given binomial is in the ideal of the tree corresponding to this division.
\end{proof}
For $n\leq 8$ we checked, using the computer programs Polymake, 4ti2, Macaulay2 and GAP, that the toric intersection of the tori of subdivisions gives the torus of the claw tree. We used the linear algebra described in the previous section. This proves that if Conjecture \ref{stop} holds for 3-Kimura model, then Conjecture \ref{hip2} holds. Moreover, in all the checked cases it was enough to consider just two subdivisions.

To summarize, we know that for 3-Kimura model Conjecture \ref{glhip} implies both Conjectures \ref{hip2} and \ref{stop} and moreover Conjecture \ref{stop} implies \ref{hip2} and for $n>8$ also Conjecture \ref{glhip}.

\subsection{Example of an application}
Let us present the method of finding phylogenetic invariants on an example. For simplicity consider the Jukes-Cantor model. Suppose that one is interested in phylogenetic invariants for $K_{4,1}$. This example is well-known and phylogenetic invariants can be found using many different methods. We have chosen this example as the number of phylogenetic invariants is small enough to be included in the paper.

Suppose that we already know the phylogenetic invariants for trivalent trees. For other group-based models, if this step is not clear, we refer to \cite{SS}. Consider a prolongation of $K_{4,1}$:
$$
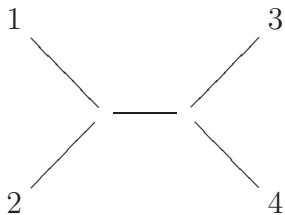
\xymatrix{
1\ar@{-}[dr]&&&\ar@{-}[dl]3\\
&\ar@{-}[r]&\ar@{-}[dr]&\\
2\ar@{-}[ur]&&&4\\
}$$
There are $8$ variables given respectively by:
$$q_{0000}, q_{0011}, q_{0101}, q_{0110}, q_{1001}, q_{1010}, q_{1100}, q_{1111}.$$
For the prolongation the ideal is generated by two relations:
$$q_{0000}q_{1111}-q_{1100}q_{0011},\quad q_{1010}q_{0101}-q_{1001}q_{0110}.$$
We may consider a different prolongation obtained by interchanging the leaves numbered $2$ and $3$. We obtain the following relations:
$$q_{0000}q_{1111}-q_{1010}q_{0101},\quad q_{1010}q_{0101}-q_{1001}q_{0110}.$$
These four (if fact one is redundant and it is enough to consider three of them) generate the ideal for $K_{4,1}$.

\section{Computational results}

\subsection{Example of a non-normal general group-based model}\label{normality}

Knowing that the projective variety associated to a general group-based model is toric, it is natural to ask whether it is normal. It is a very important property, as a lot of theorems in toric geometry work just for normal varieties. We check it investigating the polytope $P$ associated to a model: it is normal if for any natural number $n$ any point in the polytope $nP$ is a sum of $n$ points of $P$.

Computations described in \cite{mateusz} have shown that for trivalent trees for the groups $\Z_2$, $\Z_3$, $\Z_4$ and $\Z_2\times \Z_2$ the associated varieties are projectively normal. However for the 2-Kimura model the associated variety is not normal. We are interested in the question whether all models for abelian (or at least cyclic) groups are normal. Now, using our implementation of Algorithm \ref{algorytm} and Normaliz (see \cite{normaliz}) we are able to check normality for a few more models.

It is well known that toric fiber products of polytopes associated to two trees $T_1$ and $T_2$ gives a polytope associated to the tree obtained by identifying one leaf of $T_1$ and one of $T_2$ (see \cite{SS}, \cite{sull}). Note that any trivalent tree can be obtained from a tripod (a tree with one inner vertex and three leaves) by a series of such gluing procedures.
Hence, because of Lemma 5.1 from \cite{mateusz}, if we check normality for the chosen group and the tripod, we know whether all algebraic varieties for this group and any trivalent tree are normal. Using our program we can obtain the set of vertices of the polytope related to the investigated group and the tripod. Finally, we apply Normaliz \cite{normaliz} to check the normality of this polytope (in the lattice generated by its vertices). Thus we obtain

\begin{comp}\label{nienormalneobliczenia}
The polytope associated with the tripod and one of the groups $G=\Z_6, \z_8,\z_2\times\z_2\times\z_2,\z_4\times\z_2$ is not normal. Hence the algebraic variety representing this model is not normal.
\end{comp}

In particular, the class of general group-based models contains non-normal models. We believe it can be difficult to characterize the class of groups for which the associated varieties are normal, or even to determine a big (infinite) class of normal, toric general group-based models. On the other hand one has the following result:
\begin{obs}\label{nienormalneind}
Let $T$ be a phylogenetic tree and let $G_1$ be a subgroup of an abelian group $G_2$. If the variety corresponding to the tree $T$ and group $G_1$ is not normal then the variety corresponding to the tree $T$ and group $G_2$ is also not normal.
\end{obs}
\begin{proof}
Let $M_i$ be a lattice whose basis is indexed by pairs of an edge of a tree and an element of the group $G_i$. The inclusion $G_1\subseteq G_2$ gives us a natural injective morphism $f:M_1\rightarrow M_2$. Let $P_i\subset M_i$ be the polytope associated to the model for the tree $T$ and group $G_i$. Let $\tilde M_i\subset M_i$ be a sublattice spanned by vertices of the polytope $P_i$.

As $P_1$ is not normal in the lattice spanned by its vertices, there exists a point $x\in nP_1\cap \tilde M_1$, that is not a sum of $n$ vertices of the polytope $P_1$. Let us consider $y=f(x)$. The vertices of $P_1$ are mapped to vertices of $P_2$. We see that $y\in nP_2\cap \tilde M_2$. If $P_2$ was normal in $\tilde M_2$ we would be able to write $y=\sum_{i=1}^n q_i$ with $q_i\in P_2$.

Let us notice that each point in the image $f(M_1)$ has zero on each entry of the coordinates indexed by any edge and any element of the group $g\in G_2\setminus G_1$. In particular $y$ has zero on these entries. As all entries of all vertices of $P_2$ are nonnegative, this proves that all entries indexed by any edge and any element of the group $g\in G_2\setminus G_1$ are zero for $q_i$. However, we see that vertices of $P_2$ that have all non-zero entries on coordinates indexed by pairs of an edge and an element $g\in G_1$ are in the image of $P_1$. Hence $q_i=f(p_i)$ for $p_i\in P_1$. We see that $x=\sum p_i$, which is impossible.
\end{proof}

In particular we see that all abelian groups $G$ such that $|G|$ is divisible by $6$ or $8$ give rise to non-normal models.

Let $P$ be the polytope associated to the tripod and the group $\z_6$. We have already seen that $P$ is not normal, hence the associated affine variety is not normal. One would be also interested if the associated projective variety is normal or, equivalently, if the polytope is very ample. By direct computation for (any) cone associated to a vertex of a polytope $P$ we obtain the following result.
\begin{comp}
The polytope $P$ associated to the tripod and the group $\z_6$ is not very ample. Hence the associated projective toric variety is not normal.
\end{comp}
\subsection{Hilbert-Ehrhart polynomials}\label{HEpoly}

The binary Jukes-Cantor model (for trivalent trees) has an interesting property, stated and proved in \cite{BW}: an elementary mutation of a tree gives a deformation of the associated varieties (see Construction 3.23). This implies that binary Jukes-Cantor models of trivalent trees with the same number of leaves are deformation equivalent (Theorem 3.26 in \cite{BW}). As it was not obvious what to expect for other models, we computed Hilbert-Ehrhart polynomials, which are invariants of deformation, in some simple cases.

\subsubsection{Numerical results}
We checked models for two different trees with six leaves (this is the least number of leaves for which there are non-isomorphic trees, exactly two), the \emph{snowflake} and the \emph{3-caterpillar}. The most interesting ones were the cases of the biologically meaningful 2-Kimura and 3-Kimura models.

The value of the Ehrhart polynomial of a polytope $P$ for a natural number $n$ is the number of lattice points in $nP$. Thus one way to determine the Hilbert-Ehrhart polynomial of a toric variety is to compute numbers of lattice points in some multiples of the associated polytope. Even if it is not possible to get enough data to determine the polynomials (eg. because the numbers are too big), sometimes we can say that polynomials for two models are not equal, because their values for some $n$ are different.

Before we completed our computations, Kubjas computed numbers of lattice points in the third dilations of the polytopes for 3-Kimura model on the \emph{snowflake} and the \emph{3-caterpillar} with 6 leaves and got 69248000 and 69324800 points respectively (see \cite{kaie}). Thus she proved that varieties associated with these models are not deformation equivalent.

Our computations confirm her results as for the 3-Kimura model and also give the following

\begin{comp}\label{2kimura_poly}
The varieties associated with 2-Kimura models for the snowflake and the 3-caterpillar trees have different Ehrhart polynomials. In the second dilations of the polytopes there are 56992 lattice points for the snowflake and 57024 for the 3-caterpillar.

Also the pairs of varieties associated with general group-based models for the snowflake and the 3-caterpillar trees and
\begin{enumerate}
\item $G=\Z_3$,
\item $G=\Z_4$,
\item $G=\Z_5$,
\item $G=\Z_7$
\end{enumerate}
have different Hilbert-Ehrhart polynomials and therefore are not deformation equivalent. (For these groups the associated varieties are normal, which can be checked using Polymake.)
The precise results of the computations are presented in the Appendix.

In the cases of
\begin{enumerate}
\item $G=\Z_8$,
\item $G=\Z_2\times\Z_2\times \Z_2$,
\item $G=\Z_9$
\end{enumerate}
the varieties have different Hilbert functions. As we know that the models for $\Z_8$ and $\Z_2\times\Z_2\times \Z_2$ are not normal and for $\Z_9$ we could not check the normality, these results do not imply that there is no deformation equivalence in these cases.
\end{comp}

\subsubsection{Technical details}

The first attempt to compute numbers of lattice points in dilations of a polytope was the direct method: constructing the list of lattice points in $nP$ by adding vertices of $P$ to lattice points in $(n-1)P$ and reducing repeated entries. This algorithm is not very efficient, but (after adding a few technical upgrades to the implementation) we were able to confirm Kubjas' results \cite{kaie}. However, this method does not work for non-normal polytopes. As we planned to investigate 2-Kimura model, we had to implement another algorithm.

The second idea is to compute inductively the relative Hilbert polynomials, i.e. number of points in the $n$-th dilation of the polytope intersected with the fiber of the projection onto the group of coordinates that correspond to a given leaf. Our approach is quite similar to the methods used in \cite{kaie} and \cite{sull}.

First we compute two functions for the tripod.
Let $P\subset \z^{3m}\cong\z^m\times\z^m\times\z^m$ be the
polytope associated to a tripod. Let $pr_i:\z^{3m}\cong\z^m\times\z^m\times\z^m\rightarrow\z^m$ be the projection onto the $i$-th group of coordinates. We distinguish one edge of the tripod corresponding to the third group of coordinates in the lattice. Let $f$ be a function such that $f(a)$ for $a = (a_1,\ldots,a_m)\in \z^m$ is the number of lattice points in $(a_1+\dots+a_m)P$ that project to $a$ by $pr_3$. We compute $f(a)$ for sufficiently many values of $a$ to proceed with the algorithm.

\begin{exm}\label{dodwspol}
The polytope $P$ for the binary Jukes-Cantor model has the following vertices:
$$v_1=(0,1,0,1,0,1),$$
$$v_2=(0,1,1,0,1,0),$$
$$v_3=(1,0,0,1,1,0),$$
$$v_4=(1,0,1,0,0,1).$$
These are the only integral points in $P$.
In this case $f(1,0)=2$
because there are exactly two points, $(1,0,0,1,1,0)$ and $(0,1,1,0,1,0)$,
that are in $1P=P$ and project to $(1,0)$ via the third projection.
\end{exm}
The function $f$ will be our base for induction. Next, we need to compute the number of points in the fiber of a projection onto two distinguished leaves. Let $g$ be a function such that $g(a,b)$ for $(a,b) = (a_1,\dots,a_m,b_1,\dots,b_m)\in \z^m\times\z^m$ is the number of lattice points in $(a_1+\dots+a_m)P$ that project to $a$ by $pr_3$ and to $b$ by $pr_2$. We compute $g(a,b)$ for sufficiently many pairs $(a,b)$ to proceed with the algorithm.

Let $T$ be a tree with a corresponding polytope $P$ and a distinguished leaf $l$. Let $h$ be a function such that $h(a)$ for $a=(a_1,\dots,a_m)\in \z^m$ is equal to the number of points in the fiber of the projection corresponding to leaf $l$ of $(a_1+\dots+a_m)P$ onto $a$. We construct a new tree $T'$ by attaching a tripod to a chosen leaf of $T$. We call $T'$ a join of $T$ and the tripod. The chosen leaf of $T'$ will be one of the leaves of the attached tripod. As proved in \cite{BW}, \cite{SS}, \cite{mateusz}, \cite{sull} (depending on the model), the
polytope associated to a join of two trees is a fiber product of the polytopes associated to these trees. Thus we can calculate the function $h'$ for $T'$ by a following rule: $h'(a)=\sum_b g(a,b)h(b)$, where the sum is taken over all $b \in \z^m$ such that $g(a,b)\neq 0$.

This allows us to compute inductively the relative Hilbert polynomial. The last tripod could be attached in the same way. Then one obtains the Hilbert function from relative Hilbert functions simply by summing up over all possible projections. However, it is better to do the last step in a different way.

Suppose that as before we are given a tree $T$ with a distinguished leaf $l$ and a corresponding relative Hilbert function $h$. We compute the Hilbert function of the tree $T'$ that is a join of the tree $T$ and a tripod using the equality $h'(n)=\sum_a f(a)h(a)$, where $a=(a_1,\dots,a_m)$ and $\sum a_i=n$. The function $f$ is the basis for induction introduced above.

Thus, decomposing the \emph{snowflake} and the \emph{3-caterpillar} trees to joins of tripods, we can inductively compute (a few small values of) the corresponding Hilbert functions. This method works also for non-normal models, if only the Hilbert function for the tripod can be computed.
In particular, for 2-Kimura model the computations turned out to be possible, because its polytope for the tripod is quite well understood (see \cite{mateusz}, 5.4), at least to describe fully its second dilation. This way we obtained the results of \ref{2kimura_poly}.

\newpage
\section*{Appendix}

Here we present the precise results of the computations of Hilbert-Ehrhart polynomials for a few models, stated in \ref{2kimura_poly}. For each of the first groups we considered the numbers of lattice points in consecutive dilations are given.

For the groups $\z_8$, $\z_2\times\z_2\times\z_2$ and $\z_9$ we computed only the Hilbert function and, as the first two are not normal and for the last one we could not check the normality, we do not know if it is equal to the Hilbert-Ehrhart polynomial.

\subsection*{Models for $G = \Z_3$}

\begin{center}
\begin{tabular}{c|l|l}
dilation & \emph{snowflake} & \emph{3-caterpillar} \\
\hline
1 & 243 & 243 \\
2 & 21627 & 21627 \\
3 & 903187 & 904069 \\
4 & 21451311 & 21496023 \\
5 & 330935625 & 331976637 \\
6 & 3647265274 & 3662146270 \\
7 & 30770591364 & 30920349834 \\
8 & 209116329075 & 210269891871 \\
9 & 1189466778457 & 1196661601837 \\
10 & 5831112858273 & 5868930577941 \\
11 & 25205348411361 & 25377886917819 \\
\end{tabular}
\end{center}

\subsection*{Models for $G=\Z_2\times \Z_2$ (3-Kimura)}

\begin{center}
\begin{tabular}{c|l|l}
dilation & \emph{snowflake} & \emph{3-caterpillar} \\
\hline
1 & 1024 & 1024 \\
2 & 396928 & 396928 \\
3 & 69248000 & 69324800 \\
4 & 5977866515 & 5990170739 \\
5 & 291069470720 & 291864710144 \\
6 & 8967198289920 & 8995715702784 \\
\end{tabular}
\end{center}

\subsection*{Models for $G=\Z_4$}

\begin{center}
\begin{tabular}{c|l|l}
dilation & \emph{snowflake} & \emph{3-caterpillar} \\
\hline
1 & 1024 & 1024 \\
2 & 396928 & 396928 \\
3 & 69248000 & 69324800 \\
4 & 6122557220 & 6138552524 \\
5 & 310273545216 & 311525688320 \\
6 & 10009786400352 & 10062179606880 \\
\end{tabular}
\end{center}

\subsection*{Models for $G=\Z_5$}

\begin{center}
\begin{tabular}{c|l|l}
dilation & \emph{snowflake} & \emph{3-caterpillar} \\
\hline
1 & 3125 & 3125 \\
2 & 3834375 & 3834375 \\
3 & 2229584375 & 2230596875 \\
4 & 640338121875 & 642089603125 \\
\end{tabular}
\end{center}

\subsection*{Models for $G=\Z_7$}
In this case the first three dilations of the polytopes have the same number of points. The numbers of points in fourth dilations were too big to obtain precise results. Hence we computed only the numbers of points mod 64, which is sufficient to prove that the Hilbert-Ehrhart polynomials are different.

\begin{center}
\begin{tabular}{c|l|l}
dilation & \emph{snowflake} & \emph{3-caterpillar} \\
\hline
1 & 16807 & 16807 \\
2 & 117195211 & 117195211 \\
3 & 423913952448 & 423913952448 \\
4 & $\equiv 54 \mod 64$ & $\equiv 14 \mod 64$ \\
\end{tabular}
\end{center}

\subsection*{Models for $G=\Z_8$}

\begin{center}
\begin{tabular}{c|l|l}
dilation & \emph{snowflake} & \emph{3-caterpillar} \\
\hline
1 & 32768 & 32768 \\
2 & 454397952 & 454397952 \\
3 & 3375180251136 & 3375013036032 \\
\end{tabular}
\end{center}

\subsection*{Models for $G=\Z_2\times \Z_2 \times \Z_2$}

\begin{center}
\begin{tabular}{c|l|l}
dilation & \emph{snowflake} & \emph{3-caterpillar} \\
\hline
1 & 32768 & 32768 \\
2 & 454397952 & 454397952 \\
3 & 3375180251136 & 3375013036032 \\
\end{tabular}
\end{center}

\subsection*{Models for $G=\Z_9$}

\begin{center}
\begin{tabular}{c|l|l}
dilation & \emph{snowflake} & \emph{3-caterpillar} \\
\hline
1 & 59049 & 59049 \\
2 & 1499667453 & 1499667453 \\
3 & 20938605820263 & 20937202945056 \\
\end{tabular}
\end{center}

\newpage
\bibliographystyle{plain}
\bibliography{xbib}

\begin{thebibliography}{10}

\bibitem{GAP}
{GAP} --- {Groups}, {Algorithms}, and {Programming}, {Version} 4.4.10.
\newblock The GAP Group, (http://www.gap-system.org), 2007.

\bibitem{4ti2}
4ti2 team.
\newblock 4ti2---a software package for algebraic, geometric and combinatorial
  problems on linear spaces.
\newblock www.4ti2.de.

\bibitem{AllRhMarkov}
E.~S. Allman and J.~A. Rhodes.
\newblock Phylogenetic ideals and varieties for the general {Markov} model.
\newblock {\em Advances in Applied Mathematics}, 40(2):127--148, 2008.

\bibitem{normaliz}
W.~Bruns, B.~Ichim, and C.~S\"{o}ger.
\newblock Normaliz.
\newblock http://www.mathematik.uni-osnabrueck.de/normaliz.

\bibitem{WDB}
W.~Buczy\'nska, M.~Donten, and J.~A. Wi\'{s}niewski.
\newblock Isotropic models of evolution with symmetries.
\newblock {\em Contemporary Mathematics}, 496:111--132, 2009.

\bibitem{BW}
W.~Buczy\'nska and J.~A. Wi\'{s}niewski.
\newblock On geometry of binary symmetric models of phylogenetic trees.
\newblock {\em J. Eur. Math. Soc.}, 9(3):609--635, 2007.

\bibitem{CFS}
M.~Casanellas and J.~Fernandez-Sanchez.
\newblock Geometry of the {Kimura} 3-parameter model.
\newblock {\em Advances in Applied Mathematics}, 41(3):265--292, 2008.

\bibitem{Sonja}
J.~Chifman and S.~Petrovi\'{c}.
\newblock Toric ideals of phylogenetic invariants for the general group-based
  model on claw trees {$K_{1,n}$}.
\newblock {\em Proceedings of the 2nd international conference on {Algebraic}
  biology}, pages 307--321, 2007.

\bibitem{Cox}
D.~Cox, J.~Little, and H.~Schenck.
\newblock {\em Toric varieties}, volume 124 of {\em Graduate Studies in
  Mathematics}.
\newblock American Mathematical Society, 2011.

\bibitem{Draisma2012}
J.~Draisma and R.~H. Eggermont.
\newblock Finiteness results for abelian tree models.
\newblock {\em arXiv:1207.1282v1 [math.AG]}, 2012.

\bibitem{DK}
J.~Draisma and J.~Kuttler.
\newblock On the ideals of equivariant tree models.
\newblock {\em Mathematische Annalen}, 344(3):619--644, 2009.

\bibitem{4aut}
N.~Eriksson, K.~Ranestad, B.~Sturmfels, and S.~Sullivant.
\newblock Phylogenetic algebraic geometry.
\newblock {\em Projective Varieties with Unexpected Properties; Siena, Italy},
  pages 237--256, 2004.

\bibitem{ES}
S.~N. Evans and T.~P. Speed.
\newblock Invariants of some probability models used in phylogenetic inference.
\newblock {\em Ann. Statist.}, 21(1):355--377, 1993.

\bibitem{fulton}
W.~Fulton.
\newblock {\em Introduction to {Toric} {Varieties}}, volume 131 of {\em Annals
  of Mathematics Studies}.
\newblock Princeton University Press, 1993.

\bibitem{Hendy1}
M.~D. Hendy.
\newblock The relationship between simple evolutionary tree models and
  observable sequence data.
\newblock {\em Systematic Zoology}, 38:310--321, 1989.

\bibitem{Il}
N.~Ilten.
\newblock Deformations of {Rational} {Varieties} with {Codimension}-{One}
  {Torus} {Action}.
\newblock Doctoral Thesis, FU Berlin, 2010.

\bibitem{kaie}
K.~Kubjas.
\newblock Hilbert polynomial of the {Kimura} 3-parameter model.
\newblock {\em arXiv:1007.3164v1 [math.AC]}, 2010.

\bibitem{mateusz}
M.~Micha{\l}ek.
\newblock Geometry of phylogenetic group-based models.
\newblock {\em Journal of Algebra}, 339:339--356, 2011.

\bibitem{ToricIdealsGeneratedbyQuadraticBinomials}
H.~Ohsugi and T.~Hibi.
\newblock Toric ideals generated by quadratic binomials.
\newblock {\em Journal of Algebra}, 218:509--527, 1999.

\bibitem{PS}
L.~Pachter and B.~Sturmfels.
\newblock {\em Algebraic {Statistics} for {Computational} {Biology}}.
\newblock Cambridge University Press, 2005.

\bibitem{Sturmfelsksiazka}
B.~Sturmfels.
\newblock {\em Groebner bases and convex polytopes}, volume~8 of {\em
  University Lecture Series}.
\newblock American Mathematical Society, 1996.

\bibitem{SS}
B.~Sturmfels and S.~Sullivant.
\newblock Toric ideals of phylogenetic invariants.
\newblock {\em J. Comput. Biology}, 12:204--228, 2005.

\bibitem{sull}
S.~Sullivant.
\newblock Toric fiber products.
\newblock {\em J. Algebra}, 316:560--577, 2007.

\bibitem{ESS}
L.~A. Sz\'{e}kely, M.~A. Steel, and P.~L. Erd\H{o}s.
\newblock Fourier calculus on evolutionary trees.
\newblock {\em Appl. Math.}, 14(2):200--210, 1993.

\end{thebibliography}

\end{document}